\documentclass[12pt,a4paper,reqno]{article}
\usepackage{graphicx} 
\usepackage{multicol}
\usepackage{authblk}
\usepackage{graphicx}
\usepackage{amsthm}
\usepackage{amssymb,amsthm}
\usepackage{empheq,mathtools}
\usepackage{enumerate}
\numberwithin{equation}{section}
\usepackage{graphicx}
\usepackage{subcaption}
\usepackage{dsfont}
\usepackage{array}
\usepackage{rotating}
\usepackage{caption}
\usepackage{tabularx}

\makeatletter
\renewcommand{\p@subfigure}{\thefigure}
\makeatother

\providecommand{\keywords}[1]
{
	\small	
	\textbf{{Keywords --}} #1
}

\providecommand{\AMS}[1]
{
	\small	
	\textbf{{AMS subject classifications --}} #1
}

\usepackage{lineno,hyperref}
\hypersetup{
	colorlinks=true,
	linkcolor=blue,
	citecolor=magenta,
	filecolor=magenta,      
	urlcolor=blue,
	pdftitle={Overleaf Example},
	pdfpagemode=FullScreen,
}

\mathtoolsset{showonlyrefs} 

\usepackage[
backend=biber,
dashed=true,
citestyle=numeric,
bibstyle=authoryear,
sorting=nyt,
sortcites=true,
giveninits=true,
firstinits=true,
maxbibnames=99,
natbib = true
]{biblatex}

\makeatletter
\input{numeric.bbx}
\makeatother

\DeclareNameAlias{author}{given-family}
\DeclareFieldFormat[article]{title}{#1}
\DeclareFieldFormat[article]{journaltitle}{\mkbibemph{#1}}
\DeclareFieldFormat[article]{volume}{\textbf{#1}}
\DeclareFieldFormat[article]{number}{\mkbibparens{#1}}

\renewbibmacro{in:}{}

\usepackage[margin=1in]{geometry}

\newcommand{\smoothspace}{C_0^\infty(\Omega_T)}
\newcommand{\dt}{\mathrm{d}t}
\usepackage{appendix}

\newcommand{\ellip}{\mathcal{A}_\mathrm{m}}
\newcommand{\Hm}{H^{\mathrm{m}}_0(\Omega)}
\newcommand{\rmD}{\mathrm{D}^{\mathrm{m}}}

\usepackage{newtxtext}  
\usepackage{newtxmath}

\newcommand{\norm}[1]{\| #1 \|}

\theoremstyle{definition}
\newtheorem{definition}{Definition}[section]

\newtheorem{thm}{Theorem}[section]
\newtheorem{remark}{Remark}[section]
\newtheorem{lemma}{Lemma}[section]
\newtheoremstyle{case}{}{}{}{}{}{:}{ }{}
\theoremstyle{case}

\title{Convergence analysis of fourth order parabolic problems with non-linear source terms}

\begin{document}
	\newgeometry{margin=0.75in}
	\title{Well--posedness and regularity for semilinear time--dependent second and fourth order in space equations}
	
	\author[1]{Gopikrishnan Chirappurathu Remesan\thanks{mail:
			\href{mailto:gopikrishnan@iitpkd.ac.in}{gopikrishnan@iitpkd.ac.in}}}
	\date{\today}
	\affil[1]{\small{Department of Mathematics, Indian Institute of Technology Palakkad, Kanjikode, Palakkad, Kerala, 678623}}
	\maketitle
	
	\begin{abstract}
		\noindent This article discusses a unified convergence analysis of the semilinear time--dependent equation $\partial_t u + (-1)^\mathrm{m}\Delta^{\mathrm{m}}u + u^3 - u = f$ with $\mathrm{m} \in \{1,2\}$ and  homogeneous Dirichlet boundary conditions. The analysis relies on Faedo--Galerkin approximation and convergence via compactness estimates. The existence and uniqueness of the weak solution is proved when the initial data is smooth. A refined and novel analysis extends the existence result to problems with rough initial data also.
	\end{abstract}
	
	\medskip
	\noindent \keywords{Time--dependent, Semilinear equations, Fisher--Kolmogorov equation, Extended Fisher--Kolmogorov equation, Fourth--order parabolic equation, Faedo--Galerkin approximation, Convergence analysis, Compactness estimates}
	
	\medskip
	\noindent \AMS{35K58, 35K91, 35K35 (Primary), 46B50 (Secondary)}
	
	\section{Introduction}
	This article presents well--posedness and regularity results for a generic time--dependent problem that involves an $\mathrm{m}$-harmonic spatial derivative, a semilinear nonlinearity, and a square integrable source term, which reads 
	\begin{align}\label{eqn:main_problem_fourth-nonlinear}
		\partial_t u + (-1)^\mathrm{m}\Delta^{\mathrm{m}} u + \phi(u)= f\quad \mathrm{in}\;\;(0,T] \times \Omega,
	\end{align}
	wherein $\mathrm{m} \in \{1,2\}$. Here, $\Omega$ is a bounded polygonal domain in $\mathbb{R}^2$ with the Lipschitz continuous boundary $\partial \Omega$ and $0 < T < \infty$ is the final time. The unknown $u = u(t,x)$ is a function of time and space, and $\phi(u) = u^3 - u$ is the semilinear nonlinearity.
	This partial differential equation (PDE) is supplemented with the homogeneous boundary conditions 
	\begin{align} \label{eqn:homo_boundary}
		\dfrac{\partial ^j u}{\partial \boldsymbol{n}^j} = 0 \quad \text{on} \quad \partial \Omega \quad \text{for} \quad 0 \le j \le \mathrm{m}-1,
	\end{align}
	where $\boldsymbol{n}$ is the unit outward normal to $\partial \Omega.$ Moreover, the initial condition $u(0,\bullet) = u_0(\bullet)$ on $\Omega$ is specified with an a priori known function $u_0$.  When $\mathrm{m} = 1$, the PDE in~\ref{eqn:main_problem_fourth-nonlinear} reduces to the Fisher--Kolmogorov (FK) equation with the homogeneous boundary condition $u = 0$ on $\partial \Omega$. The case $\mathrm{m} = 2$ corresponds to the extended Fisher--Kolmogorov equation (EFK) with the homogeneous boundary conditions $ u = 0 = \dfrac{\partial u}{\partial \boldsymbol{n}}$  on $\partial \Omega$. 
	
	\medskip \noindent The main objective of this article is to establish an existence theory for~\eqref{eqn:main_problem_fourth-nonlinear}--\eqref{eqn:homo_boundary} under low regularity assumptions on the data. Moreover, improved spatial regularity of weak solutions using elliptic regularity is also derived. The analysis first considers the case \(u_0\in H_0^\mathrm{m}(\Omega)\), where a Faedo--Galerkin approximation combined with appropriate discrete energy estimates yields existence, uniqueness, and improved regularity. The discussion is then extended to the case \(u_0\in L^2(\Omega)\), where existence can still be proved, although the weaker time regularity changes the structure of the argument. Here, a novel key result (Lemma~\ref{lemma:L1convergence}) that connects almost everywhere and $L^p(\Omega)$ convergence is derived, which facilitates the convergence of nonlinear terms. A crucial part of the analysis is the treatment of the cubic nonlinearity through compactness and weak convergence arguments, which apply in a unified manner to both the second--order and fourth--order cases. The regularity of the data $f$ also dictates the well--posedness of the problem and regularity of the weak solution. This article assumes that $f \in L^2(0,T;L^2(\Omega))$.

	\medskip \noindent 
	
	\subsection{Physical relevance}
	The classical Fisher--Kolmogorov equation (i.e.~\eqref{eqn:main_problem_fourth-nonlinear} with \( \mathrm{m} = 1 \)) describes the propagation of a wavefront connecting the stable equilibrium state \( u = \pm 1 \) and the unstable equilibrium state \( u = 0 \). The second--order diffusion term enforces a smooth, monotone transition between these states. Figure~\ref{fig:smooth_front} illustrates this behaviour through a numerical solution of~\eqref{eqn:main_problem_fourth-nonlinear} with the initial condition \( u(0,x) = \exp(-x^2) \) and truncated boundary conditions \( u(t,0) = 1 \), \( u(t,20) = 0 \). 
	Beyond front propagation, the Fisher--Kolmogorov framework arises in several applications, including phase transition modelling in binary alloys~\cite{allencahn}, diffuse interface methods for image segmentation~\cite{Lee}, and phase--field models of tumour growth~\cite{Garcke}.

	\begin{figure}[htbp]
		\centering
		\captionsetup[subfigure]{justification=centering}
		\caption{Evolution of the numerical solution $u_h$ showing smooth monotonic front from FK and oscillatory kink from EFK equations.}
		\label{fig:front_comparison}
		\begin{subfigure}[b]{0.48\textwidth}
			\centering
			\caption{Smooth front from Fisher-Kolmogorov equation}
			\includegraphics[width=\linewidth]{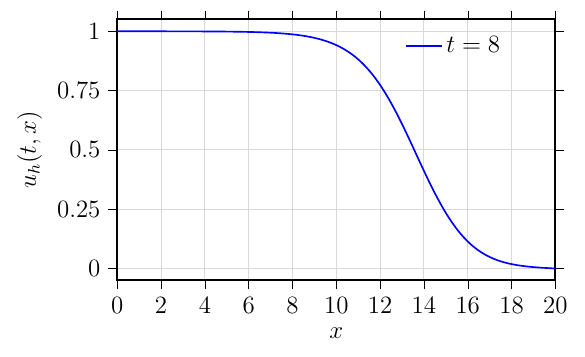}
			\label{fig:smooth_front}
		\end{subfigure}
		\hfill
		\begin{subfigure}[b]{0.48\textwidth}
			\centering
			\caption{Oscillatory kinks from extended Fisher--Kolmogorov equation}
			\includegraphics[width=\linewidth]{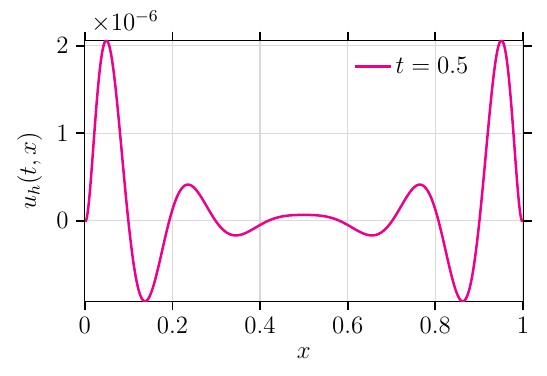}
			\label{fig:osc_kinks}
		\end{subfigure}
	\end{figure}

	\medskip  In contrast, certain bistable systems exhibit non-monotone transitions characterised by oscillatory structures, commonly referred to as \emph{kinks}. Such behaviour is captured by augmenting the Fisher--Kolmogorov equation with a higher--order diffusion term \( \gamma \Delta^2 u \), leading to the extended Fisher--Kolmogorov equation. In one spatial dimension, this takes the form 
	\begin{align} \label{eqn:1D_foruth_order}
		\dfrac{\partial u}{\partial t}
		+ \gamma \dfrac{\partial^4 u}{\partial x^4}
		- \dfrac{\partial^2 u}{\partial x^2}
		+ u^3 - u = 0.
	\end{align}
	\begin{figure}[htbp]
		\centering
		\caption{{Numerical solution $u_h$ of~\eqref{eqn:1D_foruth_order} for different values of $\gamma$ at $t = 0.5$}}
		\captionsetup{justification=centering}
		\vspace{0.5em}
		
		\begin{subfigure}{0.32\textwidth}
			\centering
			\caption{$\gamma = 0.1$}
			\label{fig:p1}
			\includegraphics[width=\linewidth]{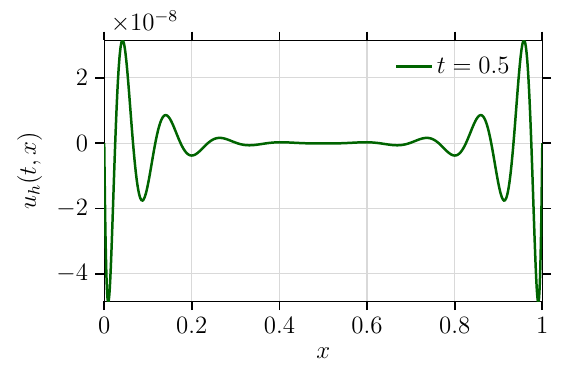}
		\end{subfigure}
		\hfill
		\begin{subfigure}{0.32\textwidth}
			\centering
			\caption{$\gamma = 0.01$}
			\label{fig:p01}
			\includegraphics[width=\linewidth]{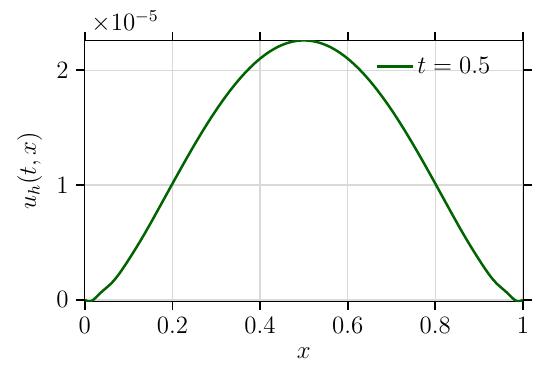}
		\end{subfigure}
		\hfill
		\begin{subfigure}{0.32\textwidth}
			\centering
			\caption{$\gamma = 0.001$}
			\label{fig:p001}
			\includegraphics[width=\linewidth]{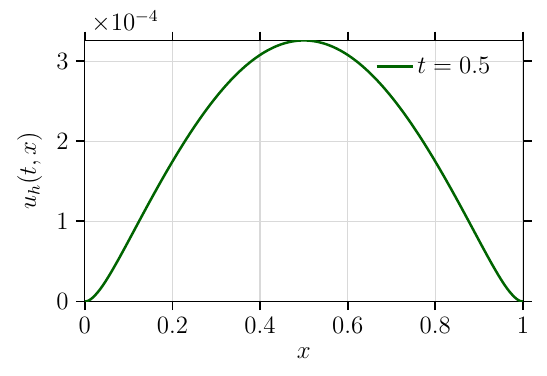}
		\end{subfigure}
	\end{figure}
	
	Figure~\ref{fig:osc_kinks} shows the emergence of oscillatory kinks around the unstable equilibrium state for \( \gamma = 1 \), with the initial condition \( u(0,x) = x^2(1-x)^2 \) and boundary conditions \( u(t,\bullet) = 0 \), \( \partial_x u(t,\bullet) = 0 \) at \(\bullet = 0,1 \) (see~\cite{pani_dhanum_1} also).
	The parameter \( \gamma \) governs the qualitative behaviour of solutions. When \( \gamma < \gamma_{\mathrm{c}} = 1/8 \) (see~\cite{Peletier}), the second-order term dominates and oscillations diminish. For sufficiently small \( \gamma \), solutions recover the monotone front structure of the Fisher--Kolmogorov equation (see Figures~\ref{fig:p01}--\ref{fig:p001}). The extended Fisher--Kolmogorov model also appears in mesoscopic model of phase transitions in binary systems~\cite{hornreich1975critical}, propagation of domain walls in liquid crystals~\cite{guozhen1982experiments}, and travelling wave phenomena in reaction--diffusion systems~\cite{aronson1978multidimensional}.
	
	\medskip
	
	The analysis in this article applies for all \( \gamma \ge 0 \). The cases \( \gamma = 0 \) and \( \gamma > 0 \) correspond to \( \mathrm{m} = 1 \) and \( \mathrm{m} = 2 \) in~\eqref{eqn:main_problem_fourth-nonlinear}, respectively.

	\subsection{Literature}
	The possibly earliest works that generalises the FK equation to the EFK equation to capture oscillatory phase transitions in bistable systems are~\cite{coulette} and~\cite{deewim}. Additionally, the authors in~\cite{deewim} explore how the behavior of solutions shift in relation to the parameter \(\gamma\) in \eqref{eqn:1D_foruth_order}. An early analysis, oscillatory features, and zeros of the solutions of FK and EFK equations are presented in~\cite{peletier,Peletier}. Several analytical properties such as maximum principles of a similar 1D problem with $\phi(u)= u(1-u)(au + b(1-u))$, where $a$ and $b$ are constants are studied in~\cite{aronson}; and its multidimensional version in~\cite{aronson1978multidimensional}.
	
	\medskip 
	The existence and uniqueness of a weak solution of~\eqref{eqn:main_problem_fourth-nonlinear} with the boundary condition $u(t,\bullet) = 0 = \Delta u(t,\bullet)$, initial condition $u_0 \in H^2(\Omega) \cap H^1_0(\Omega)$, and an optimal control that replaces $f$ is presented in~\cite{ning}. However, this analysis only supplies $u \in L^2(0,T;H^2(\Omega) \cap H^1_0(\Omega))$ and $\partial_t u \in L^2(0,T;(H^2(\Omega) \cap H^1_0(\Omega))^{\ast}).$ The Faedo--Galerkin approximation for the EFK equation (\eqref{eqn:main_problem_fourth-nonlinear}--\eqref{eqn:homo_boundary} with $\mathrm{m} = 2$), wherein $u_0 \in H^2_0(\Omega)$ that leads to existence and uniqueness of a weak solution is briefly presented in~\cite{pani_dhanum_1,pani_dhanum_2}. Regarding the temporal regularity in~\cite{pani_dhanum_2}, the inclusion $\partial_t u \in L^\infty(0,T;L^2(\Omega))$ is not immediate from the arguments presented therein. The current work fixes this and shows that $\partial_t u \in L^2(0,T;L^2(\Omega))$ rigorously.
	
	\medskip The existence of weak solutions to~\eqref{eqn:main_problem_fourth-nonlinear} with sufficient regularity plays a central role in the a priori error analysis of finite element and polytopal schemes. A mixed formulation of~\eqref{eqn:main_problem_fourth-nonlinear} subject to Neumann boundary conditions was studied in~\cite{musawi}, where continuous linear finite elements are used for spatial discretisation and the backward Euler scheme for temporal discretisation. A fully discrete error analysis for the EFK equation, based on backward Euler time stepping and nonconforming quadratic finite elements, was presented by A.~Das et al.~\cite{das_nataraj_remesan}. Their analysis assumes that 
	$u \in L^2(0,T;H^{2+\sigma}(\Omega)) \quad \text{and} \quad \partial_t u \in L^2(0,T;L^2(\Omega)).$
	Semi--discrete and fully discrete formulations employing nonconforming virtual element methods in space were investigated by L.~Pei et al.~\cite{Lifang}, under the assumption of the existence of a weak solution in the sense of Definition~\ref{defn:weak_solution}. More recently, N.~Nataraj and R.~Kumar~\cite{raman_neela} developed a fully discrete scheme in which the spatial operator is approximated using hybrid higher--order methods. The breadth of existing work from numerical analysis relying on regularity properties of the FK and EFK equations highlights the relevance and scope of the present study.

	\subsection{Organisation and contributions}
	This article is organized into two sections. Section~\ref{sec:smooth_initial_data} addresses~\eqref{eqn:main_problem_fourth-nonlinear} under the regularity assumption \(u_0 \in H^{\mathrm{m}}_0(\Omega)\) using the Faedo--Galerkin approximation, yielding existence and uniqueness of a weak solution. Though this and similar results have been reported in earlier seminal works (see, e.g.,~\cite{gudi_gupta,pani_dhanum_1,pani_dhanum_2,das_nataraj_remesan,raman_neela}), a complete and rigorous proof has not been available in the scientific literature. This work bridges this research gap.
	Section~\ref{sec:rough_initial_data} relaxes the regularity requirement to \(u_0 \in L^2(\Omega)\) and establishes the existence of a weak solution to~\eqref{eqn:main_problem_fourth-nonlinear}. The reduced regularity necessitates a more delicate analysis, relying on several key compactness and stability arguments. This is the first such result--to the best of our knowledge--in the literature establishing existence under the minimal regularity assumption \(u_0 \in L^2(\Omega)\). Section~\ref{sec:concluding remarks} presents concluding remarks and future directions. 
	
	\subsection{Preliminaries} 
	In the sequel, $\Omega_T := (0,T) \times \Omega$ and the space of all smooth functions with compact support on $\Omega_T$ is denoted by $\smoothspace$. The Sobolev--Slobedeckij space $H^k(\Omega)$, the Bochner space $L^p(0,T;X)$ with $1 \le p \le \infty$, and $C^{k}([0,T];X)$, wherein $X$ is a Banach space, bear the standard meaning (see~\cite[Chapters 5--7]{evans_PDE}) in the sequel. 
	
	\medskip \noindent The bilinear form $\ellip : \Hm \times \Hm \rightarrow \mathbb{R}$ is defined by $\ellip(u,v) = (\rmD u, \rmD v)$ for all $u, v \in \Hm.$ The standard $\|\!\bullet\!\|_{L^2(\Omega)}$ norm is abbreviated as $\|\!\bullet\!\|$ in the sequel.  The Sobolev space $V = \Hm $ is endowed with the norm $\|\!\bullet\!\|_V := \|\rmD(\bullet)\|.$
	
	\medskip \noindent Define the function $\Phi : \mathbb{R} \rightarrow \mathbb{R}$ by $\Phi(x) = (1-x^2)^2/4$ and observe that $\Phi'(x) = \phi(x) = x^3 - x$. In the sequel, we employ the Young's inequality, for $a, b >0$ and $\varepsilon \in (0,1)$ it holds
	\begin{align} \label{ineq:youngs}
		ab \le \varepsilon\dfrac{a^2}{2} + \dfrac{b^2}{2\varepsilon}.
	\end{align}
	The notation \( A \lesssim B \) means that \( A \le C B \) for a generic constant \( C>0 \), possibly depending on the domain; in particular, \( \|g\|_{X} \lesssim 1 \) for some norm $\|\!\bullet\!\|_X$ implies \( \|g\|_{X} \le C \) with \( C \) independent of \( g \).

	\begin{definition}[weak solution]  For a given initial data $u_0 \in V$, a weak solution to~\eqref{eqn:main_problem_fourth-nonlinear} is a function $u: \Omega_{T} \to  \mathbb{R}$ that satisfies 
		\begin{subequations} \label{eqn:time_dep}
			\begin{align}\label{time_dep_weak_form}
				\displaystyle (\partial_t u,w) + \ellip(u,w) + (\phi(u),w) ={}& (f,w)\quad \text{for all}\,\, w \in V, \text{ and } \\ \label{eqn:initial_data_changed}
				u_{0}(0,\bullet) ={}& u_0(\bullet).
			\end{align}   
		\end{subequations} \label{defn:weak_solution}
	\end{definition}
	\noeqref{eqn:initial_data_changed}
	\section{Convergence analysis with smooth initial data} \label{sec:smooth_initial_data}
	This section presents the convergence analysis of~\eqref{eqn:main_problem_fourth-nonlinear} for $u_0 \in V$. Section~\ref{sec:existence_uniqueness_with_V} establishes the existence and uniqueness of a weak solution using the Faedo--Galerkin approximation method. Section~\ref{sec:improved_regularity_with_V} then derives an improved regularity result for $u$ using elliptic regularity.

	\subsection{Existence and uniqueness} \label{sec:existence_uniqueness_with_V}
	\begin{thm} Suppose that $u_0 \in V$. Then, there exist a unique weak solution to~\eqref{eqn:main_problem_fourth-nonlinear} in the sense of Definition~\ref{defn:weak_solution} such that
		$u \in L^{\infty}(0,T;V)$ and $\partial_t u \in L^2(0,T;L^2(\Omega))$. \label{existencenonlinear}
	\end{thm}
	
	\medskip \noindent For technical clarity and ease, the proof is divided into steps. Step 1 presents the \emph{Faedo--Galerkin approximation} of~\eqref{time_dep_weak_form} that yields a finite dimensional approximate solution, followed by its discrete energy estimates in Step 2. The compactness results arise from these estimates are presented in Step 3. Step 4 presents convergence arguments that lead to the existence of a weak solution. Step 5 shows that the initial condition is satisfied by the weak solution, followed by the uniqueness of weak solutions in Step 6.

	\medskip \noindent \emph{Proof of Theorem~\ref{existencenonlinear}. }
	\emph{Step 1. (Finite dimensional approximation).} An application of Theorem~\ref{thm:basis} with $\mathcal{L} = (-1)^{\mathrm{m}}\Delta^\mathrm{m}$ guarantees the existence of an orthonormal basis $(w_j)_{j\geq1} \;\subset C^{\infty}(\Omega)$ of $L^2({\Omega})$. Since the functions $(w_j)_{j\geq 1}$ are eigenfunctions of the operator $(-1)^{\rm m}\Delta^{\mathrm{m}}$,  $(w_j)_{j\geq1}$ is also an orthogonal basis of $V={}\Hm$. Define the finite dimensional subspace $V_n\subset V$ as the span of $(w_j)_{1\leq j\leq n}$. 
	Suppose $u_n \in V_n$ satisfies 
	\begin{align}\label{eqn;nonlinearweakformulation}
		( \partial _tu_n,w_j)+\ellip(u_n,w_j)+(\phi(u_n),w_j)=(f(t,\bullet),w_j) \quad \text{for all}\,
		\,1 \le j \le n
	\end{align}
	and the initial condition $u_n(0,\bullet)=\pi_nu_0$, where $\pi_n:V\rightarrow V_n$ is the orthogonal projection. This is a system of ordinary differential equations with a locally Lipschitz nonlinearity $\phi(u)$. Therefore, the local existence and uniqueness of the solution to~\eqref{eqn;nonlinearweakformulation} follows from the Picard--Lindel\"{o}f  theorem (see Theorem~\ref{thm:picard_lindelof}). The global existence is guaranteed by the uniform bounds established below. Equation~\eqref{eqn;nonlinearweakformulation} implies, for $\widehat{w} \in V_n$ it holds 
	\begin{align}\label{eqn:all}
		( \partial _tu_n,\widehat{w})+\ellip(u_n,\widehat{w})+(\phi(u_n),\widehat{w})=(f(t,\bullet),\widehat{w}).
	\end{align}

	\medskip \noindent \emph{Step 2. (Discrete energy estimates).}
	Set $\widehat{w} = \partial_t u_n$ in~\eqref{eqn:all} to arrive at,
	\begin{align} \label{eqn:diff_eqn_on_un}
		\| \partial_t u_n \|^2 + \dfrac{1}{2} \dfrac{\mathrm{d}}{\mathrm{d}t} \|u_n\|_V^2  + \dfrac{\mathrm{d}}{\mathrm{d}t} (\Phi(u_n),1) = ( f(t,\bullet), \partial_t u_n ),
	\end{align}
	where $\Phi(u) = (1-u^2)^2/4$. 
	Integrate the above equation over $(0,T)$ to arrive at,
	\begin{align} \label{eqn:test_function_h1}
		\hspace{-0.75cm} \int_{0}^T\| \partial_t u_n \|^2\,\mathrm{d}t + \dfrac{1}{2}\|u_n(T, \bullet) \|_V^2  + (\Phi(u_n(T,\bullet)),1) ={}& \int_{0}^T( f(t,\bullet), \partial_t u_n)\,\mathrm{d}t +  \dfrac{1}{2}\|\pi_n u_0 \|_V^2  + (\Phi(\pi_n u_0),1)
	\end{align}
	with $u_n(0,\bullet) = \pi_n u_0$ in the last step.
	The last displayed inequality and  $\|\pi_n u_0\|_V \le \|u_0\|_V$ from Theorem~\ref{thm:orthogonal_projection} show
	\begin{align*}
		\int_{0}^T\| \partial_t u_n \|^2\,\mathrm{d}t + \|u_n(T, \bullet) \|_V^2  + (\Phi(u_n(T,\bullet)),1) 
		\lesssim{}& \frac{1}{2\varepsilon}\int_{0}^{T}\|f(t,\bullet)\|^2\;\dt\;+\frac{\varepsilon}{2}\int_{0}^{T}\|\partial_t u_n\|^2 \; \dt \\
		&+\|u_0\|_V^2  +  (\Phi(\pi_n u_0),1)
	\end{align*}
	with the Young's inequality~\eqref{ineq:youngs} in the last step.  Use the estimate $(\Phi(\pi_n u_0),1) \le (\|\pi_n u_0\|_{L^4(\Omega)}^4 + {\rm meas}(\Omega))/4$ to arrive at
	\begin{align*}
		\left(1 - \dfrac{\varepsilon}{2} \right) \int_{0}^T\| \partial_t u_n \|^2\,\mathrm{d}t + \|u_n(T, \bullet) \|_V^2  + (\Phi(u_n(T,\bullet)),1) 
		\lesssim{}& \frac{1}{2\varepsilon}\int_{0}^{T}\|f(t,\bullet)\|^2\;\dt\ \\
		&{}+ \|u_0\|_{V}^2 + (\|u_0\|_{V}^4 + {\rm meas}(\Omega))/4
	\end{align*} 
	with the inequality $\|\pi_n u_0\|_{L^4(\Omega)} \lesssim \|\pi_n u_0\|_{V}$ from the embedding $V \hookrightarrow L^q(\Omega)$ for $1 \le q < \infty$ and $\|\pi_n u_0\|_{V} \le \|u_0\|_V$ in the last step.
	Since $(\Phi(u_n(T,\bullet)),1)  \ge 0$, the last displayed inequality leads to $u_n \in L^{\infty}(0,T;V)$ and $\partial _t u_n \in L^{2}(0,T;L^2(\Omega))$. Moreover, since  inequality constant in $\lesssim$ is independent of $n\in \mathbb{N}$, the inclusions $(u_n)_{n \ge 1} \in L^{\infty}(0,T;H^{\rm m}_0(\Omega))$ and $(\partial _t u_n)_{n \ge 1} \in L^{2}(0,T;L^2(\Omega))$ are uniform. 
	
	\medskip \noindent \emph{Step 3. (Compactness).}  The above uniform estimates on $(u_n)_{n \geq 1}$ ensure the existence of a subsequence $(u_n)_{n \geq 1}$ and $u \in L^{\infty}(0,T;V)$ with $\partial_t u \in L^{2}(0,T;L^2(\Omega))$ such that
	\begin{align} \label{eqn:weak_convergences}
		u_n \xrightharpoonup {} u \;\; \mathrm{weakly}-\star \text{ in } L^{\infty}(0,T;V) \quad \text{and} \quad \partial_t u_n\xrightharpoonup {} \partial_t u \;\;\mathrm{weakly} \text{ in } L^{2}(0,T;L^2(\Omega)).
	\end{align}

	\medskip \noindent \emph{Step 4. (Convergence and existence).} Choose a function $v \in L^2(0,T;V)$ and define the orthogonal projection $v_n(t,\bullet) = \pi_n v(t,\bullet)$. Then $v_n$ converges strongly to $v$ in $L^2(0,T;V)$ (hence, $\|v_n\|_{L^2(0,T;V)} \rightarrow \|v\|_{L^2(0,T;V)}$) as $n \rightarrow \infty$ as a consequence of the Parseval's identity.   
	
	\medskip \noindent For sake of brevity, the time dependence of functions is suppressed in the sequel. For instance, $v_n(t,\bullet)$ is denoted by $v_n$. The convergence of the linear terms $\int_{0}^T (\partial_t u_n,v_n) \,\mathrm{d}t \rightarrow \int_{0}^T (\partial_t u,v) \,\mathrm{d}t$,  $\int_{0}^T\ellip(u_n,v_n)\,\mathrm{d}t \rightarrow \int_{0}^T\ellip(u,v)\,\mathrm{d}t$, and $\int_{0}^T(u_n,v_n)\,\mathrm{d}t \rightarrow \int_{0}^T(u,v)\,\mathrm{d}t$ follows from weak convergences in~\eqref{eqn:weak_convergences}. The convergence $\int_{0}^T(f,v_n)\mathrm{d}t \rightarrow \int_{0}^T(f,v)\mathrm{d}t$ follows from $v_n \rightarrow v$ strongly in $L^2(0,T;V)$. The remainder of this step establishes    
	\begin{align*}
		\int_{0}^T \int_{\Omega} u_n^{3} v_n\,\mathrm{d}x\mathrm{d}t \rightarrow \int_{0}^T \int_{\Omega} u^{3} v\,\mathrm{d}x\mathrm{d}t.
	\end{align*}
	Elementary algebra shows that 
	\begin{align*}
		\int_{0}^T \int_{\Omega} |u_n^{3} v_n - u^3 v|\,\mathrm{d}x\mathrm{d}t \le \int_{0}^T \int_{\Omega} |(u_n^{3} - u^3) v_n |\,\mathrm{d}x\mathrm{d}t + \int_{0}^T \int_{\Omega} |u^3(v_n - v) |\,\mathrm{d}x\mathrm{d}t.
	\end{align*}
	A H\"{o}lder inequality applied to the second term in the right hand side of the last displayed equation shows  
	\begin{align}
		\int_{0}^T \int_{\Omega} |u^3(v_n - v) |\,\mathrm{d}x\mathrm{d}t \le  \left( \int_{0}^T \|u\|_{L^6(\Omega)}^6\mathrm{d}t \right)^{1/2} \|v_n - v\|_{L^2(0,T;L^2(\Omega))} \lesssim   \|v_n - v\|_{L^2(0,T;L^2(\Omega))}
	\end{align}
	with the embedding $V \hookrightarrow L^6(\Omega)$ and $u \in L^{\infty}(0,T;V)$ in the last step. This and $\|v_n - v\|_{L^2(0,T;L^2(\Omega))} \rightarrow 0$ show $\int_{0}^T \int_{\Omega} |u^3(v_n - v) |\,\mathrm{d}x\mathrm{d}t \rightarrow 0$ as $n$ approaches infinity.

	\medskip \noindent In the following discussion, the convergence $\int_{0}^T \int_{\Omega} |(u_n^{3} - u^3) v_n |\,\mathrm{d}x\mathrm{d}t \rightarrow 0$ as $n \rightarrow \infty$ is proved. In Theorem~\ref{compactness theorem}, choose $X=V\;, Y=L^2(\Omega)\;\mathrm{and}\; Z=V^{\ast}$ to extract a subsequence of $(u_n)_{n \ge 1}$ such that 
	\begin{align*}
		u_n\rightarrow u\;\; \mathrm{strongly\; in}\;\; L^2(0,T;L^2(\Omega)).
	\end{align*}
	An application of the identity $a^3 - b^3 = (a-b)(a^2 + ab + b^2)$ shows 
	\begin{align*}
		\left|\int_{0}^T \int_{\Omega} (u_n^3 - u^3) v_n \,\mathrm{d}x\,\mathrm{d}t \right| \le{}& \int_{0}^T \int_{\Omega} |u_n - u| |u_n^2 + uu_n + u^2| |v_n|\,\mathrm{d}x\,\mathrm{d}t \\
		\le{}& \int_{0}^T \|u_n - u\| \|u_n^2 + uu_n + u^2\|_{L^4(\Omega)} \|v_n\|_{L^4(\Omega)}\,\mathrm{d}t. 
	\end{align*}
	The identity $(a^2+ab+b^2)^4 \le (81/2)(a^8+b^8)$ shows $\|u_n^2 + uu_n + u^2\|_{L^4(\Omega)} \lesssim (\|u_n\|_{L^8(\Omega)}^8 + \|u\|_{L^8(\Omega)}^8)^{1/4}$. This and the Sobolev embedding $V \hookrightarrow L^q(\Omega)$ for $1 \le q < \infty$ reveal 
	\begin{align} \label{eqn:roadblock_to_higherdimensions}
		\|u_n^2 + uu_n + u^2\|_{L^4(\Omega)} \lesssim (\|u_n\|_{V}^8 + \|u\|_{V}^8)^{1/4} \lesssim 1
	\end{align}
	uniformly for all $n$ and $0 < t < T$ as $u_n, u \in L^\infty(0,T;V)$.
	A similar Sobolev embedding shows that $\|v_n(t,\bullet) \|_{L^4(\Omega)} \lesssim \|v_n(t,\bullet)\|_{V}$ with a uniform constant for all $0 <  t < T$. A combination of these estimates yield 
	\begin{align*}
		\left|\int_{0}^T \int_{\Omega} (u_n^3 - u^3) v_n \,\mathrm{d}x\,\mathrm{d}t \right| \lesssim{}& \int_{0}^T \|u_n - u\|  \|v_n\|_{V}\,\mathrm{d}t \\ 
		\lesssim{}& \|u_n - u\|_{L^2(0,T;L^2(\Omega))} \|v_n\|_{L^2(0,T;V)}.
	\end{align*}
	The strong convergences $\|u_n - u\|_{L^2(0,T;L^2(\Omega))} \rightarrow 0$ and $\|v_n\|_{L^2(0,T;V)} \rightarrow \|v\|_{L^2(0,T;V)}$ show that the right hand side of the last displayed equation converges to zero. 
	
	\medskip 
	\medskip \noindent  These convergence results lead to 
	\begin{align} \label{eqn:SO_initial_cond_first}
		\int_{0}^T( \partial _tu,v)\;\dt+\int_{0}^T \ellip(u,v)\;\dt+\int_{0}^T(\phi(u),v)\;\dt =\int_{0}^T( f,v)\;\dt.
	\end{align}
	Since the last displayed equation holds for all $v \in L^2(0,T;V)$, for all $v \in V$ and almost every $0 < t < T$ it holds
	\begin{align} \label{eqn:uniqueness_so}
		( \partial _tu,v) + \ellip(u,v) + (\phi(u),v) = (f,v).
	\end{align} 
	
	\medskip \noindent \emph{Step 5. (Initial condition).} Choose $v\in C^{1}([0,T];V)$ such that $v(T,\bullet)=0$. The weak formulation~\eqref{eqn;nonlinearweakformulation} in the discrete space $V_n$ with $v_n(t,\bullet) = \pi_n v_n$ reveals
	\begin{align} \label{eqn:SO_initial_condiotion}
		\int_{0}^T \big((\partial _tu_n,v_n) + \ellip(u_n,v_n) + (\phi(u_n),v_n)\big)\;\dt=\int_{0}^T ( f,v_n)\;\dt.
	\end{align}
	Apply integration by parts to the first term in the left hand side of the last displayed equation to obtain 
	\begin{align*}
		\int_{0}^T\big( -(u_n,\partial_t v_n) + \ellip( u_n, v_n)+(\phi(u_n),v_n)\big)\;\dt=\int_{0}^T (f,v_n)\;\dt + (\pi_n u_0,v_n(0,\bullet)).
	\end{align*}
	Pass the limit $n \rightarrow \infty$ with $(\pi_n u_0,v_n(0,\bullet)) \rightarrow (u_0,v(0,\bullet))$ to arrive at   
	\begin{align}
		\int_{0}^T\big( -( u,\partial_t v) + \ellip(u,v) +(\phi(u),v) \big)\,\mathrm{d}t=\int_{0}^T (f,v)\;\dt + (u_0,v(0,\bullet)).
	\end{align}
	Since $u \in L^\infty(0,T;V) \hookrightarrow L^2(0,T;V)$ and $\partial_t u \in L^2(0,T;L^2(\Omega)) \hookrightarrow L^2(0,T;V^{\ast})$, an application of Lion--Magenes lemma (see Theorem~\ref{thm:gelfand}) implies $u \in C([0,T];L^2(\Omega))$. Therefore, the trace $u(0,\bullet) \in L^2(\Omega)$ is well defined. This and an integration by parts formula applied to the first term in the left hand side of~\eqref{eqn:SO_initial_cond_first} reveal
	\begin{align}\label{continous solution}
		\int_{0}^T \big( -(u,\partial_t v)+ \ellip(u,v)+(\phi(u),v) \big)\;\dt=\int_{0}^T(f,v)\dt + (u(0,\bullet),v(0,\bullet)).
	\end{align}
	A comparison between the last two displayed equations implies $(u(0,\bullet)-u_0,v(0,\bullet)) =0$ for all $v(0,\bullet) \in V$. Since $v(0,\bullet)$ is arbitrary, it follows $u(0,\bullet) = u_0$. 
	
	\medskip \noindent \emph{Step 6. (Uniqueness).}
	Supposing that $u_1$ and $u_2$ are two solution of~\eqref{time_dep_weak_form}, the function $z = u_1 - u_2$ satisfies
	$(\partial_t z,w)+\ellip(z,w)+(\phi(u_1) - \phi(u_2),w)=0$ for all $w \in V$. Choose $w = z(t,\bullet)$ to obtain
	\begin{align} \label{eqn:gronwall_prev}
		\frac{1}{2}\frac{\mathrm{d}}{\dt}\|z\|^2+ \|z\|_V^2 + (\phi(u_1) - \phi(u_2), u_1 - u_2) = 0.
	\end{align}
	Algebra shows that 
	\begin{align*}
		(\phi(u_1) - \phi(u_2), u_1 - u_2) = \int_{\Omega} (u_1 - u_2)^2(u_1^2 + u_1u_2 + u_2^2 - 1)\,\mathrm{d}x \ge  -\int_{\Omega} |z|^2\,\mathrm{d}x
	\end{align*}
	with $u_1^2 + u_1u_2 + u_2^2 \ge 0$ in the last step. This and~\eqref{eqn:gronwall_prev} imply 
	\begin{align*}
		\frac{1}{2}\frac{\mathrm{d}}{\dt}\|z\|^2 \le \|z\|^2.
	\end{align*}
	Then, an application of the Gronwall's inequality and $z(0,\bullet) = 0$ show $\|z(t,\bullet)\| = 0$, which leads to the uniqueness. \qed

	\begin{remark}
		Theorem~\ref{existencenonlinear} shows that a weak solution of~\eqref{eqn:main_problem_fourth-nonlinear} holds the regularity $u \in L^\infty(0,T;V)$ and $\partial_t u \in L^2(0,T;L^2(\Omega))$. Further, the Sobolev embedding $V \hookrightarrow L^q(\Omega)$ for $1 \le q < \infty$ shows $u \in L^\infty(0,T;L^q(\Omega))$. For $\mathrm{m} = 2$, the critical case $u \in L^\infty(0,T;L^\infty(\Omega))$ also holds. 
	\end{remark}
	
	
	\subsection{Improved regularity} \label{sec:improved_regularity_with_V}
	An application of the elliptic regularity allows for obtaining an improved regularity estimate on the weak solutions of~\eqref{eqn:main_problem_fourth-nonlinear}, which is presented in the next theorem. 
	
	\begin{thm}[improved regularity] 
		Suppose $u$ is a weak solution of~\eqref{eqn:main_problem_fourth-nonlinear} for $u_0 \in V$. Then, it holds $u \in L^2(0,T;H^{\mathrm{m}+\sigma}(\Omega))$ with $0 < \sigma \le 1$ that depends on the geometry of the domain. 
	\end{thm}
	\begin{proof}
		Elliptic regularity from Theorem~\ref{elliptic_regularity} with $s = \mathrm{m}  -\sigma$ establishes that the weak solution of the 
		\(\mathrm{m}\)-harmonic equation $(-1)^\mathrm{m} \Delta^\mathrm{m} u = G$ 
		satisfies the regularity estimate $\|u\|_{H^{\mathrm{m}+\sigma}(\Omega)} \le \mathrm{C}_{\mathrm{reg}}(\sigma)\, \|G\|_{H^{\sigma - \mathrm{m}}(\Omega)}$. The constants \(\mathrm{C}_{\mathrm{reg}}\) and $0 < \sigma \le 1$ only depend on the geometry of the domain \(\Omega\). The weak solution of~\eqref{eqn:main_problem_fourth-nonlinear} satisfies 
		\begin{align*}
			(-1)^\mathrm{m} \Delta^\mathrm{m} u = G = \phi(u) + f(t,\bullet) - \partial_t u.
		\end{align*}
		Therefore, it follows 
		\begin{eqnarray}\label{thm:reg_est}
			\norm{u(t,\bullet)}_{H^{\mathrm{m}+\sigma}(\Omega)} \le \mathrm{C}_{\rm reg}(\sigma)\norm{\phi(u) + f(t,\bullet) - \partial_t u}_{H^{\sigma - \mathrm{m}}(\Omega)} \lesssim  \mathrm{C}_{\rm reg}(\sigma)(\norm{\phi(u)} + \norm{f(t,\bullet)} + \norm{\partial_t u}).
		\end{eqnarray} 
		An application of the inequality $ab \le (a^2 + b^2)/2$ shows $\|\phi(u)\| \lesssim (\|u\|_{L^6(\Omega)}^6 + \|u\|^2)^{1/2}$. Then the Sobolev embedding $V \hookrightarrow L^q(\Omega)$ for $1 \le q < \infty$ implies $\|\phi(u)\| \lesssim (\|u\|_{V}^6 + \|u\|_V^2)^{1/2}$. This estimate reveals that 
		\begin{align*}
			\int_{0}^T  \norm{u(t,\bullet)}_{H^{\mathrm{m}+\sigma}(\Omega)}^2\,\mathrm{d}t \lesssim \|f\|_{L^2(0,T;L^2(\Omega))}^2 + \|\partial_t u\|_{L^2(0,T;L^2(\Omega))}^2 + \int_{0}^T (\|u\|_{V}^6 + \|u\|_V^2)\,\mathrm{d}t \lesssim 1
		\end{align*}
		with $u \in L^{\infty}(0,T;V)$ and $\partial_t u \in L^2(0,T;L^2(\Omega))$ from Theorem~\ref{existencenonlinear} and $f \in L^2(0,T;L^2(\Omega))$ in the last step.
	\end{proof}

	\begin{remark}[convex domains] If $\Omega$ is a convex polygonal domain, then $\sigma = 1$. In this case, the optimal regularity $u \in L^2(0,T;H^{\mathrm{m}+1}(\Omega))$ holds. 
	\end{remark}
	
	\begin{remark}[existence for the case $\Omega \subset \mathbb{R}^3$] For the case $\mathrm{m} = 2$, the proof readily extends to polyhedral domains in $\mathbb{R}^3$ with a Lipschitz continuous boundary $\partial \Omega.$ However, the case $\mathrm{m} = 1$ does not hold since the embedding $H^1_0(\Omega) \hookrightarrow L^8(\Omega)$ employed in~\eqref{eqn:roadblock_to_higherdimensions} is not true in $\mathbb{R}^3$. 
	\end{remark}
	
	\begin{remark}[extended Fisher--Kolmogorov equation]
		The analysis presented here for $\mathrm{m} = 2$ readily extends to Fisher--Kolmogorov equation $\partial_t u + \gamma \Delta^2 u - \Delta u + \phi(u) = f$ with very minor modifications as the Laplacian contribution is linear. The regularity of $u$ changes to $\gamma^{1/2} u \in L^\infty(0,T;V)$ and that of $\partial_t u$ remains the same.
	\end{remark}
	
	\section{Convergence analysis with minimal regularity on the initial condition} 
	\label{sec:rough_initial_data}
	Theorem~\ref{existencenonlinear} assumes the high regularity condition
	$u_0 \in V$. However, in many physical phenomena modelled
	by~\eqref{eqn:main_problem_fourth-nonlinear}, the initial datum need not satisfy this regularity. A typical example arises in a one-dimensional biphasic system, where a natural initial configuration is given by 
	\[
	u_0(x)=
	\begin{cases}
		1, & 0.25\le x\le 0.75,\\
		0, & \text{otherwise}.
	\end{cases}
	\]
	In the aforementioned example, the interval $[0.25,0.75]$ consists entirely of one phase, whereas the rest of the domain consists of the other phase (see Figure~\ref{fig:EFK_initial}). However, the numerical solution of the extended Fisher--Kolmogorov equation with this initial data is a smooth curve (see Figure~\ref{fig:EFK_final}). A similar behaviour is also demonstrated by the Fisher--Kolmogorov equation. The numerical solution of the Fisher--Kolmogorov equation with the rough initial data in Figure~\ref{fig:FK_initial} is a smooth front that connects the equilibrium states $u_h = 0$ and $u_h = 1$ (see Figure~\ref{fig:FK_final}). These numerical examples indicate the possible existence of smooth solutions even for rough initial data. 
	
	\begin{figure}[htbp]
		\centering
		
		\caption{Numerical solutions of Fisher--Kolmogorov and extended Fisher--Kolmogorov equation with a rough initial data.}
		\label{fig:FK_EFK_comparison}
		
		\vspace{0.5em}
		
		\setlength{\tabcolsep}{4pt}
		\renewcommand{\arraystretch}{1.2}
		
		\begin{tabular}{>{\centering\arraybackslash}m{0.04\textwidth}cc}
			
			\rotatebox{90}{\textcolor{blue}{Fisher--Kolmogorov}}
			&
			\begin{subfigure}[m]{0.44\textwidth}
				\centering
				\caption{Initial profile}
				\includegraphics[width=\textwidth]{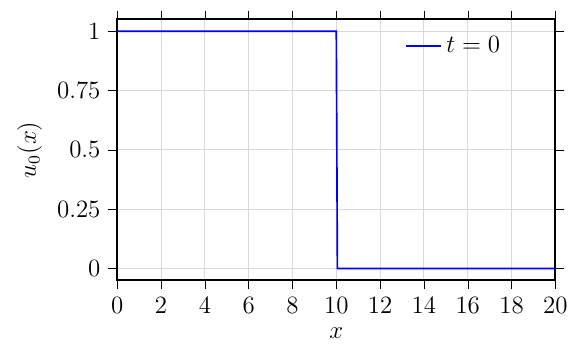}
				\label{fig:FK_initial}
			\end{subfigure}
			&
			\begin{subfigure}[m]{0.44\textwidth}
				\centering
				\caption{Numerical solution at final time $t = 2$}
				\includegraphics[width=\textwidth]{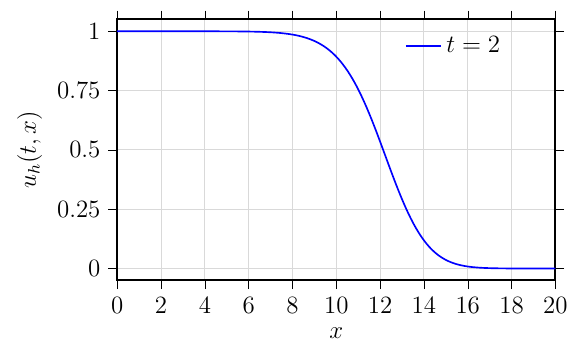}
				\label{fig:FK_final}
			\end{subfigure}
			\\[1.5em]
			
			\rotatebox{90}{\textcolor{magenta}{Ext. Fisher--Kolmogorov}}
			&
			\begin{subfigure}[m]{0.44\textwidth}
				\centering
				\caption{Initial profile}
				\includegraphics[width=\textwidth]{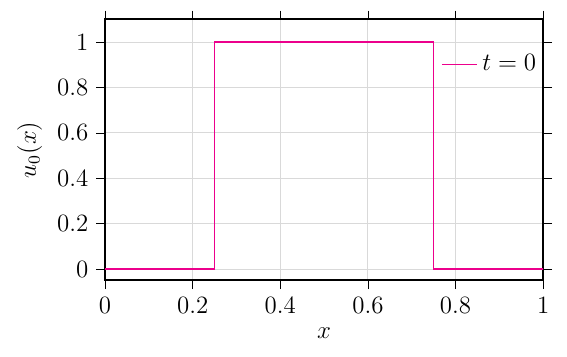}
				\label{fig:EFK_initial}
			\end{subfigure}
			&
			\begin{subfigure}[m]{0.44\textwidth}
				\centering
				\caption{Numerical solution at final time $t = 0.05$}
				\includegraphics[width=\textwidth]{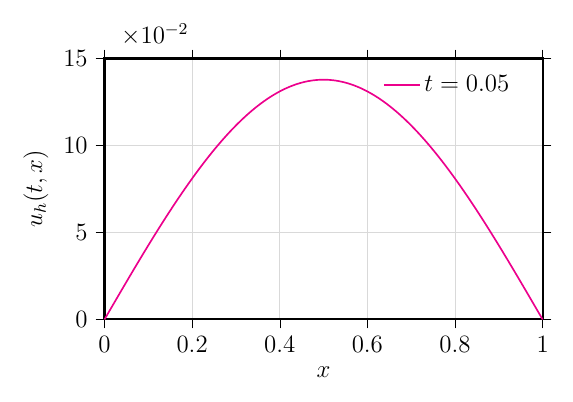}
				\label{fig:EFK_final}
			\end{subfigure}
			
		\end{tabular}
		
	\end{figure}

	\medskip \noindent Moreover, weak solutions arising from a rough initial data has a strong theoretical underpinning also. Assume that $H$ is a Hilbert space identified with its own dual $H \equiv H^{\ast}$ and the embedding $V \hookrightarrow H$ is continuous and dense; so that the embedding $H \hookrightarrow V^{\ast}$ is continuous. The chain of inclusions $V \hookrightarrow H \equiv H^{\ast} \hookrightarrow V^{\ast}$ is called a Gelfand triple. 
	\begin{thm}[Lion--Magenes Lemma~{\cite[Chapter 7]{Roubicek_nonlinearpde}}] \label{thm:gelfand}
		Suppose that  $V \hookrightarrow H \equiv H^{\ast} \hookrightarrow V^{\ast}$ is a Gelfand triple, $1 \le p \le \infty$ and $p^{\ast} = p/(p-1)$ is the conjugate exponent of $p$. Define the space 
		\begin{align*}
			W^{1,p,p^\ast}(0,T;V,V^{\ast}) = \left\{ u \in L^{p}(0,T;V):\; \dfrac{\mathrm{d}u}{\mathrm{d}t} \in L^{p^\ast}(0,T;V^{\ast}) \right\}
		\end{align*}
		Then, the embedding $W^{1,p,p^\ast}(0,T;V,V^{\ast}) \hookrightarrow C([0,T];H)$ is continuous.  
	\end{thm}
	An immediate consequence of Theorem~\ref{thm:gelfand} is that if $u \in W^{1,p,p^\ast}(0,T;V,V^{\ast})$, then $u \in C([0,T];H)$ and hence the trace $u(0,\bullet)$ is well-defined with $u(0,\bullet) \in H.$ In particular, Theorem~\ref{thm:gelfand} with $V = \Hm$ and $H = L^2(\Omega)$ shows $u \in C([0,T];L^2(\Omega))$, which leads to $u_0 = u(0,\bullet) \in L^2(\Omega).$ This suggests that the initial data $u_0$ need to possess only $L^2(\Omega)$ regularity.  
	
	\medskip \noindent Consequently, it is natural to expect the existence of a weak solution with an initial data in the space $L^2(\Omega)$, which is investigated in the sequel.  The remainder of this section is organised as follows. Section~\ref{sec:existence} presents a novel key lemma (Lemma~\ref{lemma:L1convergence}) that leads to the existence of a weak solution in Theorem~\ref{thm:low_regularity} with relaxed assumption $u_0 \in L^2(\Omega)$. Section~\ref{sec:initial_condition} establishes that the weak solution satisfies $u(0,\bullet) = u(\bullet)$ in $V^{\ast}$. The challenges in proving uniqueness of a weak solution is presented in Section~\ref{sec:lack_of_uniqueness}.

	\subsection{Existence of a weak solution} \label{sec:existence}
	\begin{lemma}[Key lemma] \label{lemma:L1convergence}
		Suppose that $E$ is an open and bounded set with $\mathrm{meas}(E) < \infty$. Let $(g_{n})_{n \ge 1}$ be a sequence of functions on $E$ such that $g_n  \xrightharpoonup{}  \widehat{g}$ weakly in $L^p(E)$ for $1 < p < \infty$ and $g_n \rightarrow g$ almost everywhere on $E$. Then, $g \in L^p(\Omega)$ and $g = \widehat{g}$ almost everywhere. 
	\end{lemma}

	\medskip \noindent \emph{Proof of Lemma~\ref{lemma:L1convergence}}.  An application of the uniform boundedness principle (see Theorem~\ref{thm:uniform_boundedness}) and the weak convergence $g_n  \xrightharpoonup{}  \widehat{g}$ imply that $\sup_{n \ge 1} \|g_{n}\|_{L^p(E)} \lesssim 1$. This, the almost everywhere convergence $g_n \rightarrow g$, and the Fatou's lemma (see Theorem~\ref{thm:fatou}) reveal
	\begin{align*}
		\int_{E} |g|^{p}\,\mathrm{d}x = \int_{E} \liminf_{n \rightarrow \infty} |g_n|^p\,\mathrm{d}x \le   \liminf_{n \rightarrow \infty} \int_{E} |g_n|^p\,\mathrm{d}x \lesssim 1.
	\end{align*}
	This concludes the proof of $g \in L^p(E).$ 
	
	\medskip \noindent 
	In what follows, we establish that $\int_{E} (\widehat{g} - g )\psi\,\mathrm{d}x = 0$ for all $\psi \in L^q(E)$ with the conjugate exponent $q = p/(1-p).$ Fix a $\psi \in L^q(\Omega)$, $\varepsilon > 0$, and an $R > 0$ which will specified later in the proof. Decompose $\psi$ as  
	\begin{align*}
		\psi = \psi^R + \psi_{R} \quad \text{ with }\quad \psi^{R} = \psi\cdot\mathds{1}_{\{|\psi| > R\}} \;\; \text{ and }\;\; \psi_{R} = \psi\cdot\mathds{1}_{\{|\psi| \le R\}} ,
	\end{align*}
	wherein $\mathds{1}_{F}$ is the indicator function of the set $F \subset E$. Then, split the integral $\int_{E} (\widehat{g} - g)\psi\,\mathrm{d}x$ as 
	\begin{align*}
		\int_{E} (\widehat{g} -g)\psi\,\mathrm{d}x =  \int_{E} (\widehat{g} - g_n)\psi\,\mathrm{d}x + \int_{E} (g_n - g)\psi^R\,\mathrm{d}x +  \int_{E} (g_n - g)\psi_R\,\mathrm{d}x =: \mathrm{I}_1 + \mathrm{I}_2 + \mathrm{I}_3.
	\end{align*}
	Since $g_n \xrightharpoonup{} g$ weakly in $L^p(E)$ and $\psi \in L^q(E)$, it follows that $\mathrm{I}_1 \rightarrow 0$ as $n \rightarrow \infty$. Therefore, for some sufficiently large $n_1 \in \mathbb{N}$, it holds $\mathrm{I}_1 \le \varepsilon/3$. A H\"{o}lder inequality shows 
	\begin{align} \label{eqn:I2}
		|\mathrm{I}_2| \le \|g_n - g\|_{L^p(E)} \|\psi^{R}\|_{L^q(E)} \lesssim \|\psi^{R}\|_{L^q(E)}
	\end{align}
	with the triangle inequality $\|g_n - g\|_{L^p(E)} \le \|g_n\|_{L^p(E)} + \|g\|_{L^p(E)}$ and $\sup_{n} \|g_n\|_{L^p(E)} \lesssim 1$ in the last step. Since $\psi \in L^q(\Omega)$, an application of the dominated convergence theorem (see Theorem~\ref{thm:DCT}) shows that $\|\psi^{R}\|_{L^q(E)}^q = \int_{\{\psi > R\}} |\psi|^q\,\mathrm{d}x$ converges to zero as $R$ approaches to infinity. Therefore, for a given $\varepsilon > 0$,  it is possible to set $R = n_2$ and for a sufficiently large $n_2 \in \mathbb{N}$,  it holds $|\mathrm{I}_2| < \varepsilon/3$.
	
	\medskip \noindent With the $R$ from the previous step, the term $I_\mathrm{3}$ can be estimated as 
	\begin{align*}
		|\mathrm{I}_3| \le \|\psi_{R}\|_{L^\infty(E)}\|g_n - g\|_{L^1(E)} \le R \|g_n - g\|_{L^1(E)} 
	\end{align*}
	with $\|\psi_{R}\|_{L^\infty(E)} \le R$ in the last step. The remainder of this step proves $\|g_n - g\|_{L^1(E)}$  converges to zero as $n$ approaches to infinity.  This is achieved through Vitali's convergence theorem (see Theorem~\ref{thm:vitali} with $p = 1$); the hypotheses to be verified are
	\begin{itemize}
		\item[(a)] for a given $\epsilon > 0$, there exists a corresponding $\delta > 0$ such that $\int_{A} |g_n|\,\mathrm{d}x \lesssim \epsilon$ for all $n$ and for all $A \subset E$ measurable with $\mathrm{meas}(A) < \delta$.
		\item[(b)] $g_n \rightarrow g$ almost everywhere.
	\end{itemize}
	Suppose that $A \subset E$ is a measurable set. An application of H\"{o}lder inequality shows
	\begin{align*}
		\int_{A} |g_n|\,\mathrm{d}x \le \|g_{n}\|_{L^p(E)} (\mathrm{meas}(A))^{1/q} \lesssim (\mathrm{meas}(A))^{1/q}
	\end{align*}
	with $\sup_n  \|g_{n}\|_{L^p(E)} \lesssim 1$. This and the choice $\delta = \epsilon^{q}$ shows $\int_{A} |g_n|\,\mathrm{d}x \lesssim \epsilon$, which verifies hypothesis (a). Hypothesis (b) follows from the common hypotheses of Lemma~\ref{lemma:L1convergence}. Therefore, Vitali's convergence theorem guarantees $\|g_n - g\|_{L^1(E)} \rightarrow 0$ as $n$ tends to infinity. Therefore, for any given $\varepsilon > 0$, it is possible to choose $n_3 \in \mathbb{N}$ sufficiently large such that $R \|g_n - g\|_{L^1(E)} \le \varepsilon/3$. This implies $\mathrm{I}_3 \le \varepsilon/3$.
	
	\medskip \noindent A combination of previous estimates reveals that, for all $n \ge \max(n_1,n_2,n_3)$ it holds $\left| \int_{E} (\widehat{g} -g)\psi\,\mathrm{d}x 
	\right| \le \varepsilon$. Proof concludes as $\varepsilon > 0$ is arbitrary.  \qed

	\begin{thm} \label{thm:low_regularity}
		Suppose that $u_0 \in L^2(\Omega)$. Then, there exists a function 
		$u \in L^\infty(0,T;L^2(\Omega)) \cap L^2(0,T;V) \cap L^4(0,T;L^4(\Omega))$ and $\partial_t u \in L^{4/3}(0,T;V^{\ast})$ such that for all $v \in V$ and almost every $t \in (0,T)$, it holds 
		\begin{align} \label{eqn:weak_form_less_regularity}
			\langle \partial_t u, v \rangle + \ellip(u,v) + (\phi(u),v) = (f(t,\bullet),v)
		\end{align}
		with the duality pairing $\langle \bullet, \bullet \rangle$ between $V^{\ast}$ and $V$.
	\end{thm}
	\medskip \noindent \emph{Proof of Theorem~\ref{thm:low_regularity}. }
	The initial step proceeds exactly as in the proof of Theorem~\ref{existencenonlinear} and is therefore omitted for brevity. The remainder of the proof is organized into four steps. Step~1 establishes the discrete energy estimates. Step~2 shows that $\partial_t u_n$ belongs to $L^{4/3}(0,T;V^{\ast}$). Step~3 is devoted to compactness and subsequential convergences. In Step~4, the passing the limit arguments are carried out, leading to the existence of a weak solution.
	
	\medskip \noindent \emph{Step 1. (Discrete energy estimates).} Set $v_n=  u_n$ in~\eqref{eqn;nonlinearweakformulation}, which then lead to 
	\begin{align} \label{eqn:low_regular_main_equation}
		\dfrac{1}{2} \dfrac{\mathrm{d}}{\mathrm{d}t} \norm{u_n}^2 + \|u_n\|_V^2 + (u_n^4,1) = ( f(t,\bullet),u_n )   + \norm{u_n}^2
	\end{align}
	with $\phi(u_n) = u_n^3 - u_n$ in the last step. The Young's inequality~\eqref{ineq:youngs} shows $( f(t,\bullet),u_n ) \lesssim \|f(t,\bullet)\|^2/2\varepsilon + \varepsilon \|u_n\|^2/2$ and $\norm{u_n}^2 \le \eta (u_n^4,1)/2 + {\rm meas}(\Omega)/2\eta$ with $0 < \epsilon, \eta \le 1$. This and the last displayed equation reveals
	\begin{align} \label{eqn:intermediate_bound}
		\dfrac{1}{2} \dfrac{\mathrm{d}}{\mathrm{d}t} \norm{u_n}^2 + \left(1 - \dfrac{\mathrm{C}^2\varepsilon}{2}\right)\|u_n\|_V^2 + \left(1 - \dfrac{\eta}{2}\right)(u_n^4,1) \le \dfrac{1}{2\varepsilon} \|f(t,\bullet)\|^2 + \dfrac{{\rm meas}(\Omega)}{2\eta} 
	\end{align}
	with $\|u_n\| \le  \mathrm{C}\|u_n\|_{V}$ in the last step. Integrate~\eqref{eqn:intermediate_bound} with a sufficiently small choice of $\eta$ and $\epsilon$ to arrive at
	\begin{align*}
		\norm{u_n(T,\bullet)}^2 + \int_{0}^T \left(\|u_n\|_V^2 + \|u_n\|_{L^4(\Omega)}^4 \right)\,\mathrm{d}t \lesssim  \int_{0}^T \|f(t,\bullet)\|^2\,\mathrm{d}t + {\rm meas}(\Omega)T + \norm{\pi_n u_0}^2.
	\end{align*}
	The observation $\norm{\pi_n u_0} \lesssim \norm{u_0}$ and the last equation supply the uniform inclusions (independent of $n$)
	\begin{align} \label{eqn:uniform_bounds}
		(u_n)_{n \ge 1} \in L^\infty(0,T;L^2(\Omega)) \cap L^2(0,T;V) \cap L^4(0,T;L^4(\Omega)).
	\end{align}
	
	\medskip \noindent \emph{Step 2. (Estimate for time derivative).} For a test function $v \in V$, Theorem~\ref{thm:orthogonal_projection} allows for the unique split $v = v_{n} + v_{n}^{\perp}$ with $v_n \in V_n$ and $v_{n}^{\perp} \in V_{n}^{\perp}$. Since $\partial_t u_n \in V_n$, it follows $(\partial_t u_n,v) = (\partial_t u_n, v_n)$. This and~\eqref{eqn;nonlinearweakformulation} defines the map $\langle \partial_t u_n, \bullet \rangle : V \rightarrow \mathbb{R}$ by 
	\begin{align} \label{eqn:time_1}
		\langle \partial_t u_n, v \rangle = (\partial_t u_n,v) = (\partial_t u_n, v_n) = (f(t,\bullet),v_n) - \ellip(u_n,v_n) - (u_n^3 - u_n,v_n).
	\end{align}
	A H\"{o}lder inequality and the embedding $V \hookrightarrow L^2(\Omega)$ from Theorem~\ref{thm:orthogonal_projection} show 
	\begin{align}  \label{eqn:time_2}
		(f(t,\bullet),v_n) - \ellip(u_n,v_n) + (u_n,v_n) \le \left( \norm{f(t,\bullet)} + 2\|u_n\|_V \right) \|v_n\|_V.   
	\end{align}
	Similarly, a H\"{o}lder inequality shows $(u_n^3,v_n) \le \|u_n\|_{L^4(\Omega)}^{3} \|v_{n}\|_{L^4(\Omega)} \lesssim \|u_n\|_{L^4(\Omega)}^{3} \|v_{n}\|_{V}$ with the embedding $V \hookrightarrow L^4(\Omega)$ in the last step. The estimates~\eqref{eqn:time_1}--\eqref{eqn:time_2} and $\|v_n\|_V \le \|v\|_V$ from Theorem~\ref{thm:orthogonal_projection} reveal
	\begin{align*}
		\langle \partial_t u_n, v \rangle \lesssim \left( \norm{f(t,\bullet)} + 2\|u_n\|_V + \|u_n\|_{L^4(\Omega)}^{3} \right) \|v\|_V, 
	\end{align*}
	which leads to $\|\partial_t u_n\|_{V^{\ast}} \le \norm{f(t,\bullet)} + 2\|u_n\|_V + \|u_n\|_{L^4(\Omega)}^{3}.$ This and the inequality $(a+b+c)^{4/3} \lesssim a^{4/3} + b^{4/3} + c^{4/3}$ with $a,b,c \ge 0$ yields $\|\partial_t u_n\|_{V^{\ast}}^{4/3} \lesssim \norm{f(t,\bullet)}^{4/3} + \|u_n\|_V^{4/3} + \|u_n\|_{L^4(\Omega)}^{4}$. Integrate this over the temporal domain $(0,T)$ and use the embedding $L^{2}(0,T;L^2(\Omega)) \hookrightarrow L^{4/3}
	(0,T;L^2(\Omega))$ to arrive at 
	\begin{align*}
		\int_{0}^T \|\partial_t u_n\|_{V^{\ast}}^{4/3}\,\mathrm{d}t \lesssim \|f\|_{L^2(0,T;L^2(\Omega))}^{4/3} +  \|u_n\|_{L^2(0,T;V)}^{4/3} + \|u_n\|_{L^4(0,T;L^4(\Omega))}^4 \lesssim 1
	\end{align*}
	from~\eqref{eqn:uniform_bounds} and $\|f\|_{L^2(0,T;L^2(\Omega))} \lesssim 1$. This shows that 
	\begin{align} \label{eqn:compactness_2}
		\partial_t u_n \in L^{4/3}(0,T;V^{\ast})
	\end{align}
	uniformly. 
	
	\medskip \noindent \emph{Step 3. (Compactness).} The uniform bounds in~\eqref{eqn:uniform_bounds} and~\eqref{eqn:compactness_2} allow for the extraction of a subsequence $(u_n)_{n \ge 1}$ and $u \in L^\infty(0,T;L^2(\Omega)) \cap L^2(0,T;V) \cap L^4(0,T;L^4(\Omega))$ with $\partial_t u \in L^{4/3}(0,T;V^{\ast})$ such that 
	\begin{gather} \label{eqn:weak_convegences_1}
		u_n \xrightharpoonup {} u \;\; \mathrm{weakly}-\star \text{ in } L^{\infty}(0,T;L^2(\Omega)), \;\; u_n\xrightharpoonup {} u \;\;\mathrm{weakly} \text{ in } L^{2}(0,T;V), \quad \text{and} \quad  \\ \label{eqn:weak_convergence_2}
		\partial_t u_n\xrightharpoonup {} \partial_t u \;\;\mathrm{weakly} \text{ in } L^{4/3}(0,T;V^{\ast}).
	\end{gather}
	Recall that $(u_n)_{n \ge 1} \subset L^2(0,T;V)$ and $(\partial_t u_n)_{n \ge 1} \subset L^{4/3}(0,T;V^{\ast})$. This and an application of Theorem~\ref{compactness theorem} with $X = V$, $Y = L^2(\Omega)$, $Z = V^{\ast}$ and $p = 4/3$ show that there exists a subsequence $(u_n)_{n \ge 1}$ and $u \in L^2(0,T;L^2(\Omega))$ such that 
	\begin{align*}
		u_n \rightarrow u \;\;\text{strongly}\;\;\text{in} \quad L^2(0,T;L^2(\Omega)).
	\end{align*}
	
	\medskip \noindent \emph{Step 3. (Convergence and existence).} Consider a test function of the form $\chi\eta_j $ with $\eta_j \in V_n$ and $\chi \in C^{\infty}(0,T).$ The general argument follows from density of $V_T = \left\{\chi\eta \;:\;\eta \in \cup_n V_n,\; \chi \in C^{\infty}(0,T) \right\}$ in $L^2(0,T;V).$ 
	
	\medskip\noindent  Convergence of $\int_{0}^T (\partial_t u_n,\eta_j)\chi\mathrm{d}t$  to $\int_{0}^T \langle \partial_t u,\eta_j \rangle \chi\mathrm{d}t$ follows from~\eqref{eqn:weak_convergence_2}. The weak convergences in~\eqref{eqn:weak_convegences_1} reveal  $\int_{0}^T \ellip(u_n,\eta_j)\chi\mathrm{d}t\rightarrow \int_{0}^T \ellip(u,\eta_j)\chi\mathrm{d}t$ and $\int_{0}^T ( u_n,\eta_j)\chi\mathrm{d}t \rightarrow \int_{0}^T ( u,\eta_j)\chi\mathrm{d}t$. 
	
	\medskip \noindent The remainder of this step endeavours in establishing 
	\begin{align} \label{eqn:weak_convegence_3}
		u_n^3 \xrightharpoonup{} u^3 \quad \text{ weakly in } \quad  L^{4/3}(0,T;L^{4/3}(\Omega)).
	\end{align}
	Since $u_n \rightarrow u$ strongly in $L^2(0,T;L^2(\Omega)),$ it follows $u_n \rightarrow u$ almost everywhere on $\Omega_T$. Since the mapping $x \mapsto x^3$ is continuous, $u_n^3 \rightarrow u^3$ almost everywhere on $\Omega_T$. Since $(u_n)_{n \ge 1} \subset L^4(0,T;L^4(\Omega))$, it follows $(u_n^3)_{n \ge 1} \subset L^{4/3}(0,T;L^{4/3}(\Omega))$. Therefore, Banach-Alaoglu theorem (see Theorem~\ref{thm:alaoglu}) establishes $u_n^3 \xrightharpoonup{} \widehat{u}$ weakly in $L^{4/3}(0,T;L^{4/3}(\Omega))$ for some $\widehat{u} \in L^{4/3}(0,T;L^{4/3}(\Omega))$. An application of Lemma~\ref{lemma:L1convergence} with $E = \Omega_{T}$, $g_n = u_n^3$, $p = 4/3$, and $\widehat{g} = \widehat{u}$ conclusively proves $\widehat{u} = u^3$ almost everywhere on $\Omega_T$. This completes proof of \eqref{eqn:weak_convegence_3}.

	\medskip \noindent As a consequence of~\eqref{eqn:weak_convegence_3}, it follows that
	\begin{align*}
		\int_{0}^T (u_n^3, \eta_j)\chi\mathrm{d}t \rightarrow \int_{0}^T (u^3, \eta_j)\chi\mathrm{d}t
	\end{align*}
	as $\eta_j\chi \in L^4(0,T;V)$, which will conclude the passage of limit. This and the aforementioned density establish for all $v \in V$ and almost every $t \in (0,T)$ 
	\begin{align*}
		\langle \partial_t u, v\rangle + \ellip(u,v) + (\phi(u),v) = (f(t,\bullet),v),
	\end{align*}
	which concludes the proof of existence. \qed
	
	\subsection{Initial condition} \label{sec:initial_condition}
	For $v_n \in V_n$ and $\chi \in C^{\infty}([0,T))$ with $\chi(T) = 0$, the Faedo--Galerkin approximation $u_n$ satisfies 
	\begin{align*}
		\int_{0}^T ((\partial_t u_n, v_n) + \ellip(u_n,v_n) + (\phi(u_n),v_n))\chi\,\mathrm{d}t = \int_{0}^T (f(t,\bullet),v_n)\chi\,\mathrm{d}t.
	\end{align*}
	An integration by parts applied to the first term in the left hand side of the last displayed equation and passing to the limit arguments show 
	\begin{align} \label{eqn:compare_1}
		-\int_{0}^T(u, v)\chi'\mathrm{d}t + \int_{0}^T  (\ellip(u,v) +(\phi(u_n),v))\chi\,\mathrm{d}t = \int_{0}^T (f(t,\bullet),v)\chi\,\mathrm{d}t + (u_0,v)\chi(0).
	\end{align}
	Further the weak solution $u$ satisfies 
	\begin{align*}
		\int_{0}^T \left(\langle \partial_t u, v \rangle  + \ellip(u,v) + (\phi(u),v)\right)\chi\,\mathrm{d}t = \int_{0}^T (f(t,\bullet),v)\chi\,\mathrm{d}t
	\end{align*}
	Theorem~\ref{thm:V*embedding} with $X = V$, $Y = V^{\ast}$, $p = 2$, $q = 4/3$ shows $u \in C([0,T];V^{\ast})$. Therefore, the trace $u(0,\bullet)$ exists in $V^{\ast}$. Consequently, an integration by parts formula applied to $\int_{0}^T \langle \partial_t u, v \rangle\,\mathrm{d}t$ in the last displayed equation reveals
	\begin{align} \label{eqn:compare_2}
		-\int_{0}^T(u, v)\chi'\mathrm{d}t + \int_{0}^T  (\ellip(u,v) + (\phi(u),v)) \chi\,\mathrm{d}t = \int_{0}^T (f(t,\bullet),v)\chi\,\mathrm{d}t + \langle u(0,\bullet),v \rangle\chi(0).
	\end{align}
	Comparing~\eqref{eqn:compare_1} and~\eqref{eqn:compare_2},  and since $\chi(0)$ is arbitrary, it follows that for all $v \in V$,  $\langle u(0,\bullet) - u_0, v \rangle = 0$ which leads to $u(0,\bullet) = u_0$ in $V^{\ast}$.
	
	\subsection{Challanges in uniqueness} \label{sec:lack_of_uniqueness}
	
	Suppose that $u_1$ and $u_2$ are two weak solutions in the sense
	of~\eqref{eqn:weak_form_less_regularity}, and let $z = u_1 - u_2$.
	Then $z$ satisfies
	\begin{align}
		\langle \partial_t z, w \rangle
		+ \mathcal{A}_\mathrm{m}(z, w)
		+ \bigl(\phi(u_1) - \phi(u_2),\, w\bigr)
		= 0
		\qquad \text{ for all }\, w \in V,\ \text{a.e.}\ t \in (0,T).
		\label{eqn:diff_eq}
	\end{align}
	The natural approach to establish uniqueness is to substitute $w = z$
	in~\eqref{eqn:diff_eq}, which would yield the energy identity
	\begin{align}
		\frac{1}{2}\frac{\mathrm{d}}{\mathrm{d}t}\|z\|^{2}
		+ \mathcal{A}_\mathrm{m}(z, z)
		+ \bigl(\phi(u_1) - \phi(u_2),\, z\bigr)
		= 0,
		\label{eqn:energy_diff}
	\end{align}
	upon invoking the identity
	\begin{align}
		2\langle \partial_t z,\, z \rangle
		= \frac{\mathrm{d}}{\mathrm{d}t}\|z\|^{2}.
		\label{eqn:int_by_parts}
	\end{align}
	However, \eqref{eqn:int_by_parts} is valid only when the map
	$t \mapsto \|z(t,\,\bullet\,)\|^{2}$ is absolutely continuous,
	which in turn requires the duality product
	$\langle \partial_t z,\, z \rangle$ to belong to
	$L^{1}(0,T)$.
	By the duality estimate and H\"{o}lder's inequality, a sufficient
	condition for this integrability is
	\begin{align*}
		\partial_t z \in L^{p}(0,T;\,V^{\ast})
		\quad \text{and} \quad
		z \in L^{p*}(0,T;\,V),
		\qquad \frac{1}{p} + \frac{1}{p*} = 1.
	\end{align*}
	The present analysis yields $z \in L^{2}(0,T;\,V)$ and
	$\partial_t z \in L^{4/3}(0,T;\,V^{\ast})$.
	Since the exponents $2$ and $4/3$ are not conjugate, the H\"{o}lder
	splitting does not produce the required $L^{1}(0,T)$ bound on
	$\langle \partial_t z,\, z \rangle$, and
	identity~\eqref{eqn:int_by_parts} cannot be justified within the
	current functional framework.
	
	\medskip \noindent 
	Suppose, for the sake of argument, that this obstruction could be
	resolved by means of a refined embedding result.
	Even granting~\eqref{eqn:int_by_parts}, an application of
	Gr\"{o}nwall's inequality would yield
	\begin{align}
		\|z(t,\,\bullet\,)\|^{2}
		\leq
		\|z(0,\,\bullet\,)\|^{2}\,\exp(c\,t),
		\qquad t \in [0,T],
		\label{eqn:gronwall}
	\end{align}
	for some constant $c > 0$.
	To conclude $z \equiv 0$, one requires
	$\|z(0,\,\bullet\,)\| = 0$, that is,
	$z(0,\,\bullet\,) = 0$ in $L^{2}(\Omega)$.
	However, as established in Section~\ref{sec:initial_condition},
	the initial condition is satisfied only in $V^{\ast}$: the comparison
	of the Galerkin equation with the limiting weak formulation yields
	solely
	\begin{align*}
		\langle z(0,\,\bullet\,),\, v \rangle = 0
		\qquad \text{ for all }\, v \in V,
	\end{align*}
	which, by the definition of the $V^{\ast}$ norm, gives
	$z(0,\,\bullet\,) = 0$ in $V^{\ast}$.
	This is strictly weaker than $z(0,\,\bullet\,) = 0$ in $L^2(\Omega)$: since
	$L^2(\Omega) \subsetneq V^{\ast}$, vanishing in $V^{\ast}$ does not imply
	vanishing in $L^2(\Omega)$, and the Gr\"{o}nwall argument~\eqref{eqn:gronwall}
	cannot be closed.
	
	\medskip
	
	Consequently, uniqueness of weak solutions in the class
	of~\eqref{eqn:weak_form_less_regularity} remains open under the
	present regularity assumptions.
	The principal obstructions are the lack of justification for
	identity~\eqref{eqn:int_by_parts} due to the non-conjugate exponents in 
	$z \in L^{2}(0,T;\,V)$ and $\partial_t z \in L^{4/3}(0,T;\,V^{\ast})$,
	and the fact that the initial condition is available only in $V^{\ast}$.
	These difficulties are intrinsic to the low regularity assumption
	$u_0 \in L^{2}(\Omega)$, which cannot be circumvented within the current framework and is the subject of an ongoing research.


	\section{Concluding remarks} \label{sec:concluding remarks}
	This article presents the existence, and uniqueness when it is possible, of weak solutions to the Fisher--Kolmogorov and extended Fisher--Kolmogorov equations under different regularity conditions on the initial data. Table~\ref{tab:regularity} concisely presents this interdependence between regularity of the initial data and that of the weak solution. A key observation from Table~\ref{tab:regularity} is that the temporal regularity of $u$ reduces with the regularity of $u_0$. The solution $u \in L^\infty(0,T;V)$ when $u_0 \in V$, whereas $u \in L^2(0,T;V)$ when $u_0 \in L^2(\Omega).$ The improved temporal regularity when $u_0 \in V$ primarily stems from using the test function $\partial_t u_n$ in~\eqref{eqn:all}. This choice transforms the term $\ellip(u_n,\partial_t u_n)$ into $\frac{\mathrm{d}}{\mathrm{d}t}\|u_n\|_V^2$. An integration of this supplies the right hand side $\|\pi_n u_0\|_V^2$ in~\eqref{eqn:test_function_h1}. This only holds when $u_0 \in V$ and leads to $u \in L^\infty(0,T;V).$ The aforementioned test function is not appropriate in the case of $u_0 \in L^2(\Omega)$. Here, the analysis is enforced to employ the test function $u_n$, which eventually lead to $u \in L^2(0,T;V).$ A possible future direction is the investigation of an interpolated regularity result on $u$ between $L^2(0,T;V)$ and $L^\infty(0,T;V)$ when $u_0 \in H^r_0(\Omega)$ with $0 \le r \le \mathrm{m}$. Another prospective research is the interpolation estimates that can be obtained through relaxing the regularity of $f$ from $L^2(\Omega)$ to $H^{-q}(\Omega)$ with $0 \le q \le \mathrm{m}$. However, these analyses are more involved and beyond the scope of the current work.

	\begin{table}[h!]
		\centering
		\begin{tabular}{|ccc|}
			\hline
			\multicolumn{3}{|c|}{Regularity}                                                  \\ \hline \hline 
			\multicolumn{1}{|c|}{$u_0$} & \multicolumn{1}{c|}{$u$} & $\partial_t u$ \\ \hline
			\multicolumn{1}{|l|}{\begin{minipage}{2cm} \centering
					$V$
				\end{minipage}
			}     & \multicolumn{1}{l|}{
				\begin{minipage}{5cm} \centering
					$L^\infty(0,T;V)$
				\end{minipage}
			}  & \begin{minipage}{3cm} \centering 
				$L^2(0,T;L^2(\Omega))$
			\end{minipage}                             \\ \hline
			\multicolumn{1}{|l|}{
				\begin{minipage}{2cm} \centering
					$L^2(\Omega)$
				\end{minipage}
			}     & \multicolumn{1}{l|}{\begin{minipage}{5cm} \centering
					$L^\infty(0,T;L^2(\Omega))$, $L^2(0,T;V)$, $L^4(0,T;L^4(\Omega))$
			\end{minipage}}  &       \begin{minipage}{3cm} \centering 
				$L^{4/3}(0,T;V^{\ast})$
			\end{minipage}                        \\ \hline
		\end{tabular}
		\caption{Regularity of initial data $u_0$, weak solution $u$, and time derivative $\partial_t u$.}
		\label{tab:regularity}
	\end{table}
	Fisher–Kolmogorov–Petrovsky–Piskunov (Fisher--KPP) equation is a similar model to that of the FK and EFK equations, wherein the nonlinearity is given by $\phi_{\rm KPP}(u) = u(u - 1).$ However, the existence and uniqueness of solutions to the Fisher--KPP equation is not properly investigated in the scientific literature. While~\cite{abdur} addresses this problem using semigroup theory and semidiscretisation methods, a detailed analysis using Faedo--Galerkin method establishing the existence of weak solutions appears to be less explored. 
	
	The nonlinearity in the FK and EFK equations $\phi(u) = u^3 - u$ admits a nonnegative potential $\Phi(u)$ with $\Phi'(u) = \phi(u)$. This and the substitution of $\widehat{w} = \partial_t u_n$ in~\eqref{eqn:all} allow for the transformation of the term $(\phi(u_n),\partial_t u_n)$ to $\dfrac{\mathrm{d}}{\mathrm{d}t} (\Phi(u_n),1)$ in~\eqref{eqn:diff_eqn_on_un}. Since $(\Phi(u_n),1) \ge 0$, the subsequent estimate~\eqref{eqn:test_function_h1} renders the discrete energy estimate 
	\begin{align} 
		\hspace{-0.75cm} \int_{0}^T\| \partial_t u_n \|^2\,\mathrm{d}t + \dfrac{1}{2}\|u_n(T, \bullet) \|_V^2  \lesssim \int_{0}^T( f(t,\bullet), \partial_t u_n)\,\mathrm{d}t +  \dfrac{1}{2}\|\pi_n u_0 \|_V^2  + (\Phi(\pi_n u_0),1)
	\end{align}
	which drives the further analysis. However, the nonlinearity in the Fisher--KPP equation does not admits such a nonnegative potetial function; instead the potential corresponding to $\phi_{\rm KPP}(u)$ is $\Phi_{\rm KPP}(u) = u^3/3 - u^2/2$, which can oscillate in sign. A similar issue is also persistent in the proof of Theorem~\ref{thm:low_regularity}. The discrete energy equation (see~\eqref{eqn:low_regular_main_equation} also) for the Fisher--KPP equation is given by
	\begin{align} 
		\dfrac{1}{2} \dfrac{\mathrm{d}}{\mathrm{d}t} \norm{u_n}^2 + \|u_n\|_V^2 + (u_n^3,1) = (f(t,\bullet),u_n )   + \norm{u_n}^2.
	\end{align}
	Since $(u_n^3,1)$ oscillates in sign, the last displayed equation fails to supply discrete energy estimates. These introduce additional analytical challenges, and a rigorous treatment of this case is the subject of an ongoing work.

	\medskip \medskip \noindent \textbf{Acknowledgement.} GCR acknowledges Prof Neela Nataraj, Department of Mathematics, Indian Institute of Technology Bombay for her support and discussions that facilitated this work.
	
	\medskip \noindent \textbf{Data availability.} Author ensure the availability of codes used in this article. Any data, specifically, Matlab codes–generated during and/or analysed during the current study will be made available on request. The simulations in this article are available in following links in Table~\ref{tab:simulation-videos}.
	\begin{table}[htbp]
		\centering
		\caption{Links to the simulation videos.}
		\label{tab:simulation-videos}
		\renewcommand{\arraystretch}{1.35}
		\begin{tabularx}{\textwidth}{|>{\centering\arraybackslash}p{1.5cm}|
				>{\centering\arraybackslash}p{4.8cm}|
				>{\raggedright\arraybackslash}X|}
			\hline
			\multicolumn{1}{|c|}{\textbf{S. No.}} 
			& \multicolumn{1}{c|}{\textbf{Content}} 
			& \multicolumn{1}{c|}{\textbf{Link}} \\
			\hline
			1 & Front propagation in FK equation
			& \href{https://www.youtube.com/watch?v=8ZJlwJgdWio}
			{https://www.youtube.com/watch?v=8ZJlwJgdWio} \\
			\hline
			2 & Kink formation in the EFK equation
			& \href{https://www.youtube.com/watch?v=I6W11E8BDOk}
			{https://www.youtube.com/watch?v=I6W11E8BDOk} \\
			\hline
			3 & Numerical solution of the FK equation with rough initial data
			& \href{https://www.youtube.com/watch?v=5BmszfDi9YI}
			{https://www.youtube.com/watch?v=5BmszfDi9YI} \\
			\hline
			4 & Numerical solution of the EFK equation with rough initial data
			& \href{https://www.youtube.com/watch?v=NeWuKbE0-R8}
			{https://www.youtube.com/watch?v=NeWuKbE0-R8} \\
			\hline
		\end{tabularx}
	\end{table}

	\printbibliography

	\newpage
	\appendix

	\section{Auxiliary definitions and results}
	This section presents some crucial definitions and results employed in the proofs in the sequel.
	
	\medskip  We now briefly recall the notion of a strongly elliptic operator with the existence of an associated orthonormal eigenbasis in \(L^2(\Omega)\). For a more comprehensive presentation, the reader can refer  to~\cite[Section~A]{folland_pde}.

	\begin{definition}[strongly elliptic operator] The differential operator $\mathcal{L} = \sum_{|\alpha| \le k} a_{\alpha}\,\partial^{\alpha}$, 
		where each coefficient \(a_{\alpha} : \Omega \to \mathbb{R}\) is smooth and \(\alpha\) is a multi--index, is called \emph{strongly elliptic of order $k$} on \(\overline{\Omega}\) if there exists a smooth complex--valued function \(\gamma\) on \(\overline{\Omega}\) with \(|\gamma(x)| = 1\), and a constant \(\mathrm{C} > 0\), such that
		\[
		\mathrm{Re}\,\!\left( \gamma(x)\sum_{|\alpha|=k} a_{\alpha}(x)\,\xi^{\alpha} \right) 
		\ge \mathrm{C}|\xi|^{k}
		\quad \text{for all } \xi \in \mathbb{R}^n \text{ and } x \in \overline{\Omega}.
		\] \label{defn:strongly_elliptic}
	\end{definition} 
	
	\begin{remark}
		Every elliptic operator with real coefficients on $\overline{\Omega}$ is strongly elliptic on $\overline{\Omega}$. In particular, the Laplacian operator $-\Delta$ and the biharomonic operator $\Delta^2$ are strongly elliptic. 
	\end{remark}
	
	\begin{thm}[Existence of orthonormal basis~{\cite[Theorem 7.23]{folland_pde}}] \label{thm:basis}
		Suppose that $\mathcal{L}$ is a strongly elliptic operator of order $2\mathrm{m}$ on $\overline{\Omega}$ that satisfies  $\mathcal{L} = \mathcal{L^*}$ . There is an orthonormal basis ${w_j}$ for $L^2(\Omega)$ consisting of eigenfunctions for $\mathcal{L}$ which are $C^{\infty}$ on $\overline{\Omega}$ and satisfy the Dirichlet condition $\partial_n^i w_j $ on $\partial\Omega$ for $0\leq i \le \mathrm{m}$. 
	\end{thm}
	
	The existence and uniqueness results for a system of ordinary differential equations is stated. For a comprehensive exposition, refer to~\cite{coddington} or \cite[Section 1.6]{Roubicek_nonlinearpde}.
	\begin{thm}[Picard--Lindelöf] \label{thm:picard_lindelof}
		Let \(\Omega \subset \mathbb{R} \times \mathbb{R}^n\) be an open set, and let 
		\(f : \Omega \to \mathbb{R}^n\) be continuous. Assume \(f\) is locally Lipschitz in the second variable. Then for any \((t_0, x_0) \in D\), there exists an interval 
		\(I = (t_0 - \delta, t_0 + \delta)\) and a unique function 
		\(x \in C^1(I;\mathbb{R}^n)\) satisfying
		\begin{align*}
			\dfrac{\mathrm{d}x}{\mathrm{d}t} = f(t,x(t)) \quad \text{ for almost every } t \in I \quad \text{and} \quad x(t_0) = x_0.
		\end{align*}
	\end{thm}
	
	Crucial results from functional analysis and measure theory are stated next.
	
	\begin{thm}[Dominated convergence theorem~{\cite[Theorem 4.2]{brezis}}] \label{thm:DCT}
		Let $(f_{n})_{n \ge 1}$ be a sequence of functions in $L^1(\Omega)$ such that $f_{n}(x) \rightarrow f(x)$ for almost every $x \in \Omega$. If there exists a function $g \in L^1$ such that for all $n$, $|f_n(x)| \le |g(x)|$ almost everywhere on $\Omega$, then $f \in L^1(\Omega)$ and $\|f - f_n\|_{L^1(\Omega)} \rightarrow 0$.
	\end{thm}
	
	\begin{thm}[Fatou's lemma~{\cite[Lemma 4.1]{brezis}}] \label{thm:fatou}
		Let $(f_{n})_{n \ge 1}$ be a sequence of functions in $L^1(\Omega)$ such that $f_{n} \ge 0$ almost everywhere and $\sup_{n} \int_{\Omega}f_n(x)\,\mathrm{d}x < \infty$ for all $n \ge 1$. Then, for almost all $x \in \Omega$, $f(x) = \liminf_{n \rightarrow \infty}f(x)$ satisfy $f \in L^1(\Omega)$ and 
		\begin{align*}
			\int_{\Omega} f(x)\,\mathrm{d}x \le \liminf_{n \rightarrow \infty} f_n(x)\,\mathrm{d}x
		\end{align*}
	\end{thm}
	
	\medskip \noindent Suppose that $E$ and $F$ are two Banach spaces and by $\mathcal{L}(E,F)$ denote the space of continuous linear operators from $E$ into $F$ equipped with the norm
	\begin{align*}
		\|T\|_{\mathcal{L}(E,F)} = \sup_{x \in E, \|x\| \le 1 } \|Tx\|
	\end{align*}
	\begin{thm}[Banach--Steinhaus, uniform boundedness principle~{~\cite[Theorem 2.2]{brezis}}] \label{thm:uniform_boundedness}
		Let $E$ and $F$ are Banach spaces. Suppose that $(T_{\alpha})_{\alpha \in \mathfrak{A}}$ be a family of continuous linear operators from $E$ to $F$ with some arbitrary indexing set $\mathfrak{A}$ (not necessarily countable). Assume that $\sup_{\alpha} \|T_{\alpha}x\| < \infty$ for all $x \in E$. Then it holds, $\sup_{\alpha} \|T\|_{\mathcal{L}(E,F) } < \infty.$
	\end{thm}
	
	
	\begin{thm}[Banach--Alaoglu~{\cite[Theorem 3.18]{brezis}}] \label{thm:alaoglu}
		If $E$ is a reflexive Banach space and $(x_n)_{n \ge 1}$ be a bounded sequence in $E$, then there exists a subsequence $(x_{n_k})_{k \ge 1}$ that is weakly convergent. 
	\end{thm}
	
	\begin{thm}[Vitali's convergence theorem{~\cite[Chapter 4, Exercises 4.14]{brezis}}] \label{thm:vitali}
		Let $(f_n)_{n \ge 1}$ be a sequence in $L^p(\Omega)$ with $1 \le p < \infty$. Under the following hypothesis:
		\begin{itemize}
			\item[(a)] for all $\epsilon > 0$, there exists $\delta > 0$ such that $\int_{A} |f_n|^p\,\mathrm{d}x < \epsilon$ for all $n$ and for all $A \subset \Omega$ measurable with $|A| < \delta$.
			\item[(b)] $f_n \rightarrow f$ almost everywhere.
		\end{itemize}
		then, it holds $f \in L^p(\Omega)$ and $f_n \rightarrow f$ in $L^p(\Omega).$
	\end{thm}
	\begin{thm}[Compactness~{\cite[Lemma 7.7]{Roubicek_nonlinearpde}}]\label{compactness theorem}
		Suppose that $X \xhookrightarrow{}Y\xhookrightarrow{}Z$ are Banach spaces.,where $X$ and $Z$ are reflexive and $X$ is compactly embedded in $Y$. If the functions $u_n: (0,T)\rightarrow X$ are such that $(u_n)_{n \ge 1}$ is uniformly bounded in $L^2(0,T;X)$ and $(\partial_t u_{n})_{n \ge 1}$ is uniformly bounded in $L^p(0,T;Z)$, then there is a subsequence that converges strongly in $L^2(0,T;Y)$. 
	\end{thm}
	
	\begin{thm}[Orthogonal projection theorem~{\cite[Corollary 5.4]{brezis}}] \label{thm:orthogonal_projection}
		Let \(H\) be a Hilbert space and let \(M \subset H\) be a closed subspace. Then for every \(v \in H\), there exists a unique element \(\pi_M v \in M\) such that $v - \pi_M v \in M^\perp.$ Equivalently, every \(v \in H\) admits a unique decomposition $v = \pi_M v + (v - \pi_M v)$
		with \(\pi_M v \in M\) and \(v - \pi_M v \in M^\perp\). Moreover, the mapping \(\pi_M : H \to M\) is linear and satisfies the Pythagoras rule $$\|v\|^2 = \|\pi_M v\|^2 + \|v - \pi_M v\|^2$$ and in particular $\|\pi_M v\| \le \|v\|$
	\end{thm}
	
	\begin{thm}[Elliptic regularity~{\cite[Section 2.1]{CC_NN}}] \label{elliptic_regularity}
		The weak solution $u \in H^\mathrm{m}_0(\Omega)$ to the problem $(-1)^\mathrm{m} \Delta^\mathrm{m} u = G$ with $G \in H^{-s}(\Omega)$, wherein $\mathrm{m} - \sigma \le s \le \mathrm{m}$ satisfies $u \in V \cap H^{2\mathrm{m}-s}(\Omega)$ and $\|u\|_{H^{2\mathrm{m} - s}(\Omega)} \le \mathrm{C}_{\rm reg}(s) \|G\|_{H^{-s}(\Omega)}$. Here, $\sigma$ is a constant the exclusively depend on the geometry of the domain $\Omega$ and $0 < \sigma \le 1$.
	\end{thm}

	\begin{thm}[{\cite[Lemma 7.1]{Roubicek_nonlinearpde}}] \label{thm:V*embedding}
		Suppose that $p,q \ge 1$, $X$ is a Banach space, $Y$ is a locally convex space, and $X \hookrightarrow Y$ continuously. Then, the space $\{u: u \in L^p(0,T;X),\partial_t u \in L^q(0,T;Y) \}$ is continuously embedded in $C([0,T];Y)$. 
	\end{thm}
	
	\begin{thm}[{\cite[Chapter 5, Theorem 4]{evans_PDE}}] \label{thm:improved_reg_thm}
		Suppose that $\Omega$ is an open and bounded set with sufficiently smooth domain $\partial \Omega$ and $k \in \mathbb{N}$. If $u \in L^2(0,T;H^{k+2}(\Omega))$ and $\partial_t u \in L^2(0,T;H^k(\Omega))$, then $u \in C([0,T],H^{k+1}(\Omega)).$ 
	\end{thm}

\end{document}